\documentclass[11pt]{article}
\usepackage[a4paper,margin=1in]{geometry}
\usepackage{amsmath,amssymb,amsthm,mathtools}
\usepackage{enumitem}
\usepackage{microtype}
\usepackage{indentfirst}
\usepackage[numbers,sort&compress]{natbib}
\usepackage[hidelinks]{hyperref}
\hypersetup{
  pdftitle={The A-alpha-Spectral Turan Problem for s-Clique Tensors of K-(r+1)-Free Graphs},
  pdfauthor={Xiaoqi Liu and Haiying Shan},
  pdfsubject={Spectral extremal problems for clique tensors},
  pdfkeywords={alpha-s-clique spectral radius,
    spectral Turan problem, complete multipartite graph, Zykov symmetrization}
}
\newtheorem{theorem}{Theorem}[section]
\newtheorem{lemma}[theorem]{Lemma}
\newtheorem{proposition}[theorem]{Proposition}
\newtheorem{corollary}[theorem]{Corollary}
\theoremstyle{definition}

\newtheorem{problem}[theorem]{Open Problem}
\theoremstyle{remark}

\newcommand{\Cs}{\mathcal C_s}
\newcommand{\As}{\mathcal A_s}
\newcommand{\Ds}{\mathcal D_s}
\newcommand{\Aas}{\mathcal A_{\alpha,s}}
\newcommand{\rhoas}{\rho_{\alpha,s}}

\title{The \(\mathcal A_\alpha\)-Spectral Tur\'an Problem for
\(s\)-Clique Tensors of \(K_{r+1}\)-Free Graphs}
\author{Xiaoqi Liu, \, Haiying Shan\thanks{E-mail:
\texttt{2111152@tongji.edu.cn} (X.~Liu),
\texttt{shan\_haiying@tongji.edu.cn} (H.~Shan, corresponding author).}
\\ \small School of Mathematical Sciences, Key Laboratory of Intelligent
Computing and Applications\\ \small (Ministry of Education), Tongji
University, Shanghai 200092, China}
\date{}

\begin{document}
\maketitle

\begin{abstract}
For fixed $r\geq3$ and $0<\alpha<1$, we study the maximum
$\alpha$-$s$-clique spectral radius among $n$-vertex $K_{r+1}$-free graphs. Here
$\mathcal A_{\alpha,s}(G)=\alpha\mathcal D_s(G)+(1-\alpha)\mathcal A_s(G)$,
where $\mathcal A_s(G)$ and $\mathcal D_s(G)$ are the adjacency and degree
tensors of the $s$-clique hypergraph of $G$.
For $s=r$, put $c_r=((r-1)/r)^{r-1}$. We prove that, for all sufficiently
large $n$, the unique extremal graph is the Tur\'an graph $T_r(n)$
if $\alpha\leq c_r$, and $K_1\vee T_{r-1}(n-1)$ if $\alpha>c_r$.
For general $3\leq s\leq r$ and $n\geq r$, every extremal graph is
complete $r$-partite. The asymptotic maximum
$\alpha$-$r$-clique spectral radius
among $n$-vertex $K_{r+1}$-free graphs is
$\max\{r^{1-r},\alpha(r-1)^{1-r}\}n^{r-1}+o(n^{r-1})$.
In every asymptotically extremal sequence of complete $r$-partite graphs,
all parts have size $n/r+o(n)$ when $\alpha<c_r$; when $\alpha>c_r$,
one part has size $o(n)$ and the remaining parts have size
$n/(r-1)+o(n)$. At $\alpha=c_r$, both part-size patterns are asymptotically extremal.
\end{abstract}



\section{Introduction}

Mantel~\cite{Mantel} determined the maximum number of edges in a
triangle-free graph. Tur\'an~\cite{Turan} extended this result to
$K_{r+1}$-free graphs, with the Tur\'an graph $T_r(n)$
as the unique extremal graph. For $n\geq r$, $T_r(n)$ is the complete
$r$-partite graph on $n$ vertices whose part sizes differ by at most one.
Generalized Tur\'an problems count copies of
a fixed graph instead of edges. Zykov~\cite{Zykov} and
Erd\H{o}s~\cite{ErdosCliques} determined the maximum number of $s$-cliques in
a $K_{r+1}$-free graph for $2\leq s\leq r$; Alon and
Shikhelman~\cite{AlonShikhelman} studied the problem for general pairs of
graphs.

Spectral extremal graph theory replaces the number of edges by the
adjacency spectral radius. Nosal~\cite{Nosal} proved the spectral Mantel
theorem, and Nikiforov~\cite{Nikiforov2007} extended it to $K_{r+1}$-free
graphs. Bollob\'as and Nikiforov~\cite{BollobasNikiforov} obtained spectral
bounds involving clique counts.

The tensor approach extends spectral graph theory to uniform hypergraphs.
Qi~\cite{Qi2005} and Lim~\cite{Lim2005} introduced tensor eigenvalues.
Cooper and Dutle~\cite{CooperDutle} defined the adjacency tensor of a uniform
hypergraph. Liu and Bu~\cite{LiuBu} applied this construction to the
$s$-cliques of a graph. Yu and Peng~\cite{YuPengZykov} further studied
the spectral Tur\'an problem for $s$-clique tensors.
Bu, Liu, and Zeng~\cite{BuLiuZengLocalization} proved a
localization theorem: their upper bounds on the $s$-clique spectral radius
use the numbers of vertices in the largest cliques containing each
vertex or $s$-clique. Xu and Fan~\cite{XuFanDisjointCliques} studied the
$s$-clique spectral Tur\'an problem when the forbidden graph is a disjoint
union of copies of $K_{r+1}$.

Nikiforov~\cite{Nikiforov} introduced the matrix
$A_\alpha(G)=\alpha D(G)+(1-\alpha)A(G)$, which combines the adjacency and
degree matrices. Suppose $F$ has chromatic number $\chi(F)=r+1\geq3$ and is
color-critical.
For~each~fixed $0\leq\alpha<1-1/r$,
Wu, Lu, and Kang~\cite[Theorem~11]{WuLuKangAlphaTuran} proved that
$T_r(n)$ is the unique extremal $F$-free graph for all sufficiently large $n$.

Guo and Zhou~\cite{GuoZhou} studied the corresponding
$\mathcal A_\alpha$-tensor for uniform hypergraphs. Wang, Shan, and
Wang~\cite{WangShanWang} studied further properties of the $\alpha$-spectral radius.
The signless Laplacian tensor is
$\mathcal Q=\mathcal D+\mathcal A=2\mathcal A_{1/2}$.
Lu, Wu, and Kang~\cite{LuWuKangSignless} gave a general criterion for
signless Laplacian spectral Tur\'an problems for uniform hypergraphs.

Let $\mathcal T=(t_{i_1\cdots i_s})$ be an order-$s$, dimension-$n$
complex tensor, where $s\geq2$ and $n\geq1$.
Following Qi~\cite{Qi2005} and Lim~\cite{Lim2005}, a number
$\lambda\in\mathbb C$ is an eigenvalue of $\mathcal T$ if there is a
nonzero vector $x\in\mathbb C^n$ such that
\[
 \sum_{i_2,\ldots,i_s=1}^n
 t_{i i_2\cdots i_s}x_{i_2}\cdots x_{i_s}
 =\lambda x_i^{s-1},\qquad i=1,\ldots,n.
\]
The sums on the left form the vector $\mathcal T x^{s-1}$.
The spectral radius of $\mathcal T$ is
\[
 \rho(\mathcal T):=\max\{|\lambda|:\lambda\text{ is an eigenvalue of }
 \mathcal T\}.
\]

A hypergraph $\mathcal H=(V(\mathcal H),E(\mathcal H))$ consists of a finite
vertex set $V(\mathcal H)$ and an edge set $E(\mathcal H)$ whose elements
are subsets of $V(\mathcal H)$. It is $s$-uniform if every edge has
exactly $s$ vertices~\cite{CooperDutle}.
Let $\mathcal H$ be an $s$-uniform hypergraph on $\{1,\ldots,n\}$.
Cooper and Dutle~\cite{CooperDutle} defined its adjacency tensor
$\mathcal A(\mathcal H)=(a_{i_1\cdots i_s})$ as the symmetric tensor
of order $s$ and dimension $n$ with entries
\[
 a_{i_1\cdots i_s}=
 \begin{cases}
 \dfrac1{(s-1)!},&\{i_1,\ldots,i_s\}\in E(\mathcal H),\\[2mm]
 0,&\text{otherwise}.
 \end{cases}
\]
The degree of vertex $i$ is $d_{\mathcal H}(i)=|\{e\in E(\mathcal H):i\in e\}|$.
The degree tensor $\mathcal D(\mathcal H)$ is the diagonal tensor of
order $s$ and dimension $n$ with diagonal entries
$d_{\mathcal H}(1),\ldots,d_{\mathcal H}(n)$~\cite{QiLaplacian}.
For $0<\alpha<1$, the $\mathcal A_\alpha$-tensor of
$\mathcal H$~\cite{GuoZhou} is defined by
\[
 \mathcal A_\alpha(\mathcal H)
 =\alpha\mathcal D(\mathcal H)+(1-\alpha)\mathcal A(\mathcal H).
\]
Its spectral radius is denoted by $\rho_\alpha(\mathcal H)$ and is
called the $\alpha$-spectral radius of $\mathcal H$.
We study the $\mathcal A_\alpha$-tensor of the $s$-clique hypergraph of a
$K_{r+1}$-free graph. Let $\Cs(G)$ be the set of $s$-cliques of $G$,
and define its $s$-clique hypergraph by
\[
 \mathcal H_s(G):=(V(G),\Cs(G)).
\]
For $\mathcal H_s(G)$, write $\As(G)=\mathcal A(\mathcal H_s(G))$ and
$\Ds(G)=\mathcal D(\mathcal H_s(G))$. The $s$-clique degree of a vertex
$v$ is $d_s^G(v)=d_{\mathcal H_s(G)}(v)$. Set
\[
 \Aas(G)=\mathcal A_\alpha(\mathcal H_s(G)),
 \qquad \rhoas(G)=\rho_\alpha(\mathcal H_s(G)).
\]
We call $\rhoas(G)$ the $\alpha$-$s$-clique spectral radius of $G$.
If $s>r$, $\Cs(G)$ is empty for every $K_{r+1}$-free graph,
so $\rhoas(G)=0$.

At $\alpha=0$, $\mathcal A_{\alpha,s}(G)=\As(G)$.
For $n\geq r$, Liu and Bu~\cite{LiuBu} proved that $T_r(n)$ is the unique
$n$-vertex $K_{r+1}$-free graph with maximum $r$-clique spectral radius.
Yu and Peng~\cite[Theorem~10]{YuPengZykov} extended this extremal
characterization to the $s$-clique spectral radius for every $2\leq s\leq r$.

When $s=2$, $\mathcal A_{\alpha,2}(G)=A_\alpha(G)$.
For $n\geq r$, Nikiforov~\cite[Theorem~27]{Nikiforov} determined all
$n$-vertex $K_{r+1}$-free graphs with maximum $A_\alpha$-spectral radius.
The unique extremal graph is $T_r(n)$ when $0\leq\alpha<1-1/r$, and
$K_{r-1}\vee\overline K_{n-r+1}$ when $1-1/r<\alpha<1$.
At $\alpha=1-1/r$, the extremal graphs are exactly the complete $r$-partite graphs.

\noindent\begin{minipage}{\linewidth}
\setlength{\parindent}{1.5em}
Our main result gives a complete classification for $s=r$ and all
sufficiently large $n$. Put $c_r=((r-1)/r)^{r-1}$.

\begin{theorem}[Spectral Tur\'an theorem for the $\alpha$-$r$-clique tensor]
\label{thm:alpha-clique-turan}
Let $r\geq3$ be an integer and $0<\alpha<1$.
There exists an integer $n_0=n_0(r,\alpha)\geq r$ such that, for every
$n\geq n_0$, the
unique $n$-vertex $K_{r+1}$-free graph
with maximum $\alpha$-$r$-clique spectral radius is
\[
 \begin{cases}
 T_r(n),&0<\alpha\leq c_r,\\
 K_1\vee T_{r-1}(n-1),&c_r<\alpha<1.
 \end{cases}
\]
\end{theorem}
\end{minipage}

To prove Theorem~\ref{thm:alpha-clique-turan}, we first prove that every
extremal graph is complete $r$-partite. We then analyze the asymptotic
maximum and limiting part sizes. The next two theorems state the structural
result and the asymptotic maximum.

\begin{theorem}\label{thm:multipartite-reduction}
Let $r\geq3$, $2\leq s\leq r$, and $n\geq r$ be integers, and let
$0<\alpha<1$. If $G$ is an
$n$-vertex $K_{r+1}$-free graph with maximum $\alpha$-$s$-clique spectral
radius, then $G$ is complete $r$-partite.
\end{theorem}

For $s=r$, the asymptotic analysis gives the following limit for the
maximum $\alpha$-$r$-clique spectral radius.

\begin{theorem}\label{thm:global-phase}
For every fixed $r\geq3$ and $0<\alpha<1$,
\begin{equation}\label{eq:global-phase}
\lim_{n\to\infty}n^{1-r}
 \max_{\substack{|V(G)|=n\\G\text{ is }K_{r+1}\text{-free}}}
 \rho_{\alpha,r}(G)
 =\max\left\{\frac1{r^{r-1}},
               \frac{\alpha}{(r-1)^{r-1}}\right\}.
\end{equation}
\end{theorem}

The two terms in the maximum in Theorem~\ref{thm:global-phase} come from
$T_r(n)$ and $K_1\vee T_{r-1}(n-1)$, respectively.
Proposition~\ref{prop:multipartite-phase} describes the structure of
asymptotically extremal complete $r$-partite graphs. Below $c_r$, all
parts have size $n/r+o(n)$. Above $c_r$, one part has size $o(n)$ and
the remaining parts have size $n/(r-1)+o(n)$. At $c_r$, both part-size
patterns are asymptotically extremal.
These asymptotic part-size conclusions are used in
Lemma~\ref{lem:eventual-part-sizes} to determine the maximizing integer part sizes.

For $0<\alpha\leq c_r$, the conclusion of
Theorem~\ref{thm:alpha-clique-turan} extends to every positive $n$ divisible by $r$.

\begin{corollary}\label{cor:balanced-multipartite}
Let $r\geq3$ and $n\geq1$ be integers with $r\mid n$, and let
$0<\alpha\leq c_r$.
Every $n$-vertex $K_{r+1}$-free graph $G$ satisfies
\[
 \rho_{\alpha,r}(G)
 \leq \rho_{\alpha,r}(T_r(n))
 =\left(\frac nr\right)^{r-1},
\]
with equality if and only if $G\cong T_r(n)$.
\end{corollary}

\section{Preliminaries}\label{sec:preliminaries}

Throughout this paper, all graphs are finite, simple, and undirected, and
uniqueness is understood up to isomorphism.
The vertex and edge sets of a graph $G$ are
$V(G)$ and $E(G)$, and $G[W]$ is
the subgraph induced by $W\subseteq V(G)$.
For disjoint graphs $G$ and $H$, their join $G\vee H$ is obtained
by adding all edges between their vertex sets.
We use $\mathbb R_+=[0,\infty)$ and $\mathbb R_{>0}=(0,\infty)$.
Empty sums and products are $0$ and $1$, respectively.

\noindent\begin{minipage}{\linewidth}
\setlength{\parindent}{1.5em}
An $s$-uniform hypergraph is $d$-regular if every vertex has degree $d$.
A walk in $\mathcal H$ is an alternating sequence
$v_0,e_1,v_1,\ldots,e_\ell,v_\ell$ of vertices
$v_i\in V(\mathcal H)$ and edges $e_i\in E(\mathcal H)$ such that
$\{v_{i-1},v_i\}\subseteq e_i$ for $1\leq i\leq\ell$.
The hypergraph $\mathcal H$ is connected if any two distinct vertices
are joined by a walk. The hypergraph $\mathcal H_s(G)$ may be disconnected
even when $G$ is connected.
\end{minipage}

For a symmetric nonnegative tensor $\mathcal T$ of order $s$, a nonnegative
eigenvector $x$ corresponding to $\rho(\mathcal T)$ and satisfying
$\sum_i x_i^s=1$ is called a Perron vector of $\mathcal T$~\cite{ChangPearsonZhang,Qi2013}.

The associated directed graph of $\mathcal T$ has vertex set
$\{1,\ldots,n\}$ and an arc $i\to j$ for distinct $i,j$ whenever
$t_{i i_2\cdots i_s}\ne0$ for some indices with
$j\in\{i_2,\ldots,i_s\}$. A tensor is called \emph{weakly irreducible}
if this directed graph is strongly connected.

For $0<\alpha<1$, the tensors $\mathcal A_\alpha(\mathcal H)$ and
$\mathcal A(\mathcal H)$ have the same associated directed graph.
This graph is strongly connected exactly when $\mathcal H$ is connected.

\begin{lemma}[Perron--Frobenius theorem; \cite{FriedlandGaubertHan}]
\label{lem:perron-frobenius}
Let $\mathcal T$ be a weakly irreducible nonnegative tensor of order $\ell\geq2$.
Then $\rho(\mathcal T)$ has a unique positive eigenvector $x$ satisfying
$\sum_i x_i^\ell=1$.
\end{lemma}

For a connected $s$-uniform hypergraph $\mathcal H$, the unique positive
Perron vector of $\mathcal A_\alpha(\mathcal H)$ is called the
$\alpha$-Perron vector of $\mathcal H$~\cite{GuoZhou}.
For a graph $G$ with connected $\mathcal H_s(G)$,
the same term refers to the $\alpha$-Perron vector of $\mathcal H_s(G)$.

An automorphism of $\mathcal H$ is a permutation of its vertices that
preserves its edge set. By Lemma~\ref{lem:perron-frobenius}, every automorphism
of a connected $\mathcal H$ preserves its unique $\alpha$-Perron vector.

\begin{lemma}[\cite{ShaoShanZhang}]\label{lem:block-decomposition}
Let $\mathcal H$ be a finite $s$-uniform hypergraph with $s\geq2$ and
at least one vertex, and let $0<\alpha<1$. If its connected components
are $\mathcal H_1,\ldots,\mathcal H_t$, then
\[
 \rho_\alpha(\mathcal H)=\max_{1\leq i\leq t}\rho_\alpha(\mathcal H_i).
\]
\end{lemma}

\begin{lemma}[\cite{Qi2013,KannanShakedMondererBerman}]
\label{lem:subhypergraph-monotonicity}
Let $\mathcal H$ be a finite $s$-uniform hypergraph with $s\geq2$, and let
$0<\alpha<1$. Let $\mathcal K$ be a subhypergraph of $\mathcal H$ with
at least one vertex. Then
\[
 \rho_\alpha(\mathcal K)\leq\rho_\alpha(\mathcal H).
\]
If $\mathcal H$ is connected and $\mathcal K\ne\mathcal H$, then the
inequality is strict.
\end{lemma}

\begin{lemma}[{\cite[Theorem~2]{Qi2013}}]
\label{lem:variational}
Let $\mathcal H$ be a finite $s$-uniform hypergraph with $s\geq2$ and
at least one vertex, and let $0<\alpha<1$. Then
\begin{equation}\label{eq:hypergraph-variational}
 \rho_\alpha(\mathcal H)
 =\max_{\substack{x\in\mathbb R_+^{V(\mathcal H)}\\\sum_vx_v^s=1}}
 P_\alpha(\mathcal H,x),
\end{equation}
where
\[
 P_\alpha(\mathcal H,x)
 =\alpha\sum_{v\in V(\mathcal H)}d_{\mathcal H}(v)x_v^s
 +(1-\alpha)s\sum_{e\in E(\mathcal H)}\prod_{v\in e}x_v.
\]
For any Perron vector $x$ of $\mathcal A_\alpha(\mathcal H)$ and every
$v\in V(\mathcal H)$, the eigenvalue equation is
\begin{equation}\label{eq:hypergraph-coordinate}
 \rho_\alpha(\mathcal H)x_v^{s-1}
 =\alpha d_{\mathcal H}(v)x_v^{s-1}
 +(1-\alpha)\sum_{\substack{e\in E(\mathcal H)\\v\in e}}
 \prod_{w\in e\setminus\{v\}}x_w.
\end{equation}
In particular, for a graph $G$, write
\begin{equation}\label{eq:variational}
 P_{\alpha,s}(G,x):=P_\alpha(\mathcal H_s(G),x)
 =\alpha\sum_{v\in V(G)}d_s^G(v)x_v^s
 +(1-\alpha)s\sum_{C\in\Cs(G)}\prod_{v\in C}x_v.
\end{equation}
Then $\rhoas(G)=\max_{x\geq0,\,\sum_vx_v^s=1}P_{\alpha,s}(G,x)$, and
the coordinate equation becomes
\begin{equation}\label{eq:clique-coordinate}
 \rhoas(G)x_v^{s-1}
 =\alpha d_s^G(v)x_v^{s-1}
 +(1-\alpha)\sum_{\substack{C\in\Cs(G)\\v\in C}}
 \prod_{w\in C\setminus\{v\}}x_w.
\end{equation}
\end{lemma}

The positive-vector bounds below follow from the Collatz--Wielandt
formula.

\begin{lemma}[\cite{FriedlandGaubertHan}]\label{lem:positive-vector-bounds}
Let $\mathcal H$ be a finite connected $s$-uniform hypergraph with
$s\geq2$ and at least one edge, and let $0<\alpha<1$. For every positive
vector $z$ indexed by $V(\mathcal H)$,
\begin{equation}\label{eq:positive-vector-bounds}
 \min_{v\in V(\mathcal H)}
 \frac{(\mathcal A_\alpha(\mathcal H)z^{s-1})_v}{z_v^{s-1}}
 \leq\rho_\alpha(\mathcal H)\leq
 \max_{v\in V(\mathcal H)}
 \frac{(\mathcal A_\alpha(\mathcal H)z^{s-1})_v}{z_v^{s-1}}.
\end{equation}
\end{lemma}

The degree bounds below follow from the row-sum bounds. The strict
inequality follows from~\eqref{eq:hypergraph-coordinate} at a vertex
of maximum degree, since the $\alpha$-Perron vector is positive.

\begin{lemma}[\cite{Qi2013}]\label{lem:hypergraph-bounds}
Let $\mathcal H$ be a finite connected $s$-uniform hypergraph with
$s\geq2$ and at least one edge, and let $0<\alpha<1$.
Let $\overline d=s|E(\mathcal H)|/|V(\mathcal H)|$ be its average degree
and $\Delta=\max_vd_{\mathcal H}(v)$ its maximum degree. Then
\begin{equation}\label{eq:degree-bounds}
 \overline d\leq\rho_\alpha(\mathcal H)\leq\Delta,
 \qquad \rho_\alpha(\mathcal H)>\alpha\Delta.
\end{equation}
In particular, if $\mathcal H$ is $d$-regular, then
$\rho_\alpha(\mathcal H)=d$.
\end{lemma}

The polynomial $P_{\alpha,s}(G,x)$ in~\eqref{eq:variational} also makes sense at
$\alpha=0$ and $\alpha=1$, where it is associated with $\As(G)$ and
$\Ds(G)$, respectively. We use these endpoints only for interpolation
in the proof of Proposition~\ref{prop:multipartite-phase}. The extremal
problem assumes $0<\alpha<1$.

\section{Proofs of the main results}

\subsection{Reduction to complete multipartite graphs}

Let $u,v$ be distinct nonadjacent vertices of a graph $H$.
Copying $u$ to $v$ means deleting all edges incident with $v$ and
joining $v$ to every neighbor of $u$. Denote the resulting graph by
$H_{v\to u}$.

To prove Theorem~\ref{thm:multipartite-reduction}, we use Zykov symmetrization
to copy vertices. The following lemma ensures that this
operation preserves $K_{r+1}$-freeness.

\begin{lemma}[Zykov symmetrization~\cite{Zykov}]
\label{lem:zykov-symmetrization}
Let $r\geq1$ be an integer, and let $u,v$ be distinct nonadjacent
vertices of a graph $H$. If $H$ is $K_{r+1}$-free, then so is
$H_{v\to u}$.
\end{lemma}

\begin{proof}[Proof of Theorem~\ref{thm:multipartite-reduction}]
Let $G$ be an extremal graph. Since $T_r(n)$ is $K_{r+1}$-free and
contains an $s$-clique, $\rho_{\alpha,s}(G)>0$.

We first prove that $\mathcal H_s(G)$ is connected, which gives the
positive Perron vector needed for the vertex comparisons. Suppose otherwise.
By Lemma~\ref{lem:block-decomposition}, choose the vertex set $W$ of a
component whose $\alpha$-spectral radius equals $\rhoas(G)$, and put $H=G[W]$.
No $s$-clique meets two components of $\mathcal H_s(G)$, so
$\mathcal H_s(H)$ is the selected component, with the same vertex degrees.
Thus $\mathcal H_s(H)$ is connected and has positive $\alpha$-spectral radius.
Choose an $s$-clique $C=\{v_1,\ldots,v_s\}$ of $H$
and a vertex $w\notin W$. Form $H'$ from $H$ by adding $w$ and joining
it to $v_1,\ldots,v_{s-1}$. The graph $H'$ is $K_{r+1}$-free, since
$w$ has degree $s-1\leq r-1$. The $s$-clique hypergraph $\mathcal H_s(H')$ is
connected and contains $\mathcal H_s(H)$ as a proper subhypergraph.
Lemma~\ref{lem:subhypergraph-monotonicity} therefore gives
\[
 \rhoas(H')>\rhoas(H)=\rhoas(G).
\]
Adding $n-|W|-1$ isolated vertices to $H'$ gives an $n$-vertex
$K_{r+1}$-free graph with $\alpha$-$s$-clique spectral radius $\rhoas(H')$, contradicting
extremality. Hence $\mathcal H_s(G)$ is connected.

Let $G_0$ be obtained from $G$ by deleting all edges contained in no
$s$-clique. This deletion changes neither $\mathcal A_{\alpha,s}(G)$ nor
$\mathcal H_s(G)$. Hence $G_0$ is extremal, $\mathcal H_s(G_0)$ is
connected, and every edge of $G_0$ lies in an $s$-clique. By
Lemma~\ref{lem:perron-frobenius}, $G_0$ has an $\alpha$-Perron vector~$x$.

Copying a vertex can change the clique degrees of several vertices. To include
all these changes, group both terms of~\eqref{eq:variational} by cliques
and assign each $s$-clique $C$ the weight
\[
 q_x^{(s)}(C)=\alpha\sum_{v\in C}x_v^s
 +(1-\alpha)s\prod_{v\in C}x_v.
\]
For $v\in V(G_0)$, define
\[
 L_s(v)=\sum_{\substack{C\in\Cs(G_0)\\v\in C}}q_x^{(s)}(C),
 \qquad Z_s(v)=L_s(v)-\rho_{\alpha,s}(G_0)x_v^s.
\]
Thus $L_s(v)$ is the total contribution to~\eqref{eq:variational}
from cliques containing $v$. For distinct nonadjacent vertices $u,v$,
apply Lemma~\ref{lem:zykov-symmetrization} to $G_0$ and denote the
resulting $K_{r+1}$-free graph by $G_{v\to u}$.

Let $y$ be obtained from $x$ by replacing $x_v$ with $x_u$, and put
$D_{uv}=\sum_wy_w^s=1-x_v^s+x_u^s$. The $s$-cliques of $G_{v\to u}$
containing $v$ correspond, by replacing $v$ with $u$, to those of $G_0$
containing $u$. Their total contribution at $y$ is $L_s(u)$.
No clique contains both $u$ and $v$, so
\[
 P_{\alpha,s}(G_{v\to u},y)
 =\rho_{\alpha,s}(G_0)-L_s(v)+L_s(u).
\]
Since $P_{\alpha,s}(G_{v\to u},y)$ is homogeneous of degree $s$ in $y$,
\eqref{eq:hypergraph-variational}, applied to
$\mathcal H_s(G_{v\to u})$ and $y/D_{uv}^{1/s}$, gives
\[
 \rho_{\alpha,s}(G_{v\to u})-\rho_{\alpha,s}(G_0)
 \geq
 \frac{L_s(u)-L_s(v)+\rho_{\alpha,s}(G_0)(x_v^s-x_u^s)}{D_{uv}}
 =\frac{Z_s(u)-Z_s(v)}{D_{uv}}.
\]
Since $G_0$ is extremal and $D_{uv}>0$, we have
$Z_s(u)\leq Z_s(v)$. Interchanging
$u$ and $v$ gives $Z_s(v)\leq Z_s(u)$. Hence
\[
 Z_s(u)=Z_s(v)\qquad\text{whenever }uv\notin E(G_0).
\]

We now use the identity $Z_s(u)=Z_s(v)$ for nonadjacent $u,v$ to
prove that nonadjacency is transitive. Suppose that there are vertices $i,j,k$ with $ij,jk\notin E(G_0)$ and
$ik\in E(G_0)$. Apply Lemma~\ref{lem:zykov-symmetrization} twice,
first copying $j$ to $i$, then $j$ to $k$. Neither
operation changes the neighborhood of $j$. Denote the resulting
$K_{r+1}$-free graph by $G''$. The vertices $i,j,k$ are pairwise
nonadjacent. For each of $i$ and $k$, replacing that vertex with $j$
gives a bijection from the new $s$-cliques containing it to the old
$s$-cliques containing $j$.
Let $y$ be
obtained from $x$ by replacing both $x_i$ and $x_k$ with $x_j$, and put

\[
 D=\sum_w y_w^s=1-x_i^s-x_k^s+2x_j^s>0.
\]

Define
\[
 L_s(i,k)=\sum_{\substack{C\in\Cs(G_0)\\i,k\in C}}q_x^{(s)}(C).
\]
The old cliques containing $i$ or $k$ contribute
$L_s(i)+L_s(k)-L_s(i,k)$. The new cliques containing $i$ or $k$ contribute
$2L_s(j)$. Therefore
\[
 P_{\alpha,s}(G'',y)
 =\rho_{\alpha,s}(G_0)-L_s(i)-L_s(k)+L_s(i,k)+2L_s(j).
\]
Using $Z_s(i)=Z_s(j)=Z_s(k)$ and applying~\eqref{eq:hypergraph-variational}
to $y/D^{1/s}$, we obtain
\[
 \rho_{\alpha,s}(G'')-\rho_{\alpha,s}(G_0)
 \geq\frac{-Z_s(i)-Z_s(k)+2Z_s(j)+L_s(i,k)}{D}
 =\frac{L_s(i,k)}{D}>0,
\]
since $ik$ lies in an $s$-clique and $x>0$. This contradicts extremality.
Thus nonadjacency is transitive. Declare each vertex nonadjacent
to itself. Nonadjacency is then an equivalence relation. Each equivalence
class is an independent set, and vertices in different classes are adjacent. Hence $G_0$ is complete
multipartite.

Let $t$ be the number of parts of $G_0$. Since $G_0$ contains an $s$-clique
and is $K_{r+1}$-free, $s\leq t\leq r$. Suppose $t<r$. As $n\geq r$,
some part has at least two vertices. Split one vertex from this part into
a new part, and denote the resulting complete $(t+1)$-partite graph by
$G_1$. It is $K_{r+1}$-free, and $\mathcal H_s(G_1)$ is connected and
contains $\mathcal H_s(G_0)$ as a proper subhypergraph.
Lemma~\ref{lem:subhypergraph-monotonicity} gives
$\rhoas(G_1)>\rhoas(G_0)$, a contradiction. Hence $t=r$.
Finally, if an edge of $G$ was deleted in forming $G_0$, it must lie within
one of the $r$ parts. Its endpoints, together with one vertex from each of
the other $r-1$ parts, form a $K_{r+1}$ in $G$, a contradiction. Therefore
$G=G_0$ is complete $r$-partite.
\end{proof}

\subsection{Asymptotic optimization of the part sizes}

From this subsection onward, we assume $1\leq n_1\leq\cdots\leq n_r$
for $K_{n_1,\ldots,n_r}$. Permuting the part sizes gives an isomorphic
graph with the same $\alpha$-$s$-clique spectral radius.

We first compute the $\alpha$-$r$-clique spectral radius when all but one part have equal size.
Assume $s=r$. For positive integers $m\leq q$, let
\[
H_{m,q}:=\overline K_m\vee T_{r-1}((r-1)q).
\]
Thus $H_{m,q}$ is the complete $r$-partite graph with part sizes
$m,q,\ldots,q$.

\begin{proposition}
\label{prop:multipartite}
Let $r\geq3$ and $0<\alpha<1$, and let $m\leq q$ be positive integers.
Put $u=m/q$. Let $\tau$ be the unique solution in $(0,1]$ of
\begin{equation}\label{eq:tau}
(1-\alpha)\tau^r+\alpha\tau
=
u\bigl(\alpha\tau+1-\alpha\bigr).
\end{equation}
Then
\begin{equation}\label{eq:multipartite-radius}
\rho_{\alpha,r}(H_{m,q})
=
q^{r-1}
\bigl(\alpha+(1-\alpha)\tau^{r-1}\bigr).
\end{equation}
In particular, for every fixed $m$,
\begin{equation}\label{eq:fixed-m-asymptotic}
\rho_{\alpha,r}(H_{m,q})
=
\alpha q^{r-1}
+
\frac{(1-\alpha)^r m^{r-1}}{\alpha^{r-1}}
+O_{\alpha,r,m}(q^{-1}),
\qquad (q\to\infty).
\end{equation}
\end{proposition}

\begin{proof}
The $r$-clique hypergraph $\mathcal H_r(H_{m,q})$ is connected.
Its $\alpha$-Perron vector is invariant under automorphisms and has a common coordinate $a$ on
the part of size $m$ and a common coordinate $b$ on the parts of size $q$. Put $\lambda=\rho_{\alpha,r}(H_{m,q})$ and $\tau=b/a$.
By~\eqref{eq:clique-coordinate}, the two eigenvalue equations are
\[
 \lambda a^{r-1}=\alpha q^{r-1}a^{r-1}
 +(1-\alpha)q^{r-1}b^{r-1},
\]
and
\[
 \lambda b^{r-1}=\alpha m q^{r-2}b^{r-1}
 +(1-\alpha)m q^{r-2}ab^{r-2}.
\]
Dividing the two equations by $a^{r-1}$ and $b^{r-1}$, respectively, gives
\[
 \lambda=q^{r-1}\bigl(\alpha+(1-\alpha)\tau^{r-1}\bigr),
 \qquad
 \lambda=mq^{r-2}\bigl(\alpha+(1-\alpha)/\tau\bigr).
\]
Equating the two expressions for $\lambda$ gives
$(1-\alpha)\tau^r+\alpha\tau=u(\alpha\tau+1-\alpha)$.
To prove that this equation has a unique solution in $(0,1]$, let
\[
 g(\tau)=(1-\alpha)\tau^r+\alpha\tau
 -u(\alpha\tau+1-\alpha).
\]
For every $\tau>0$,
\[
 g'(\tau)=r(1-\alpha)\tau^{r-1}+\alpha(1-u)>0.
\]
Hence $g$ is strictly increasing on $(0,\infty)$. Since
$g(0)=-u(1-\alpha)<0$ and $g(1)=1-u\geq0$, it has exactly one positive
zero, which lies in $(0,1]$.

For fixed $m$, we have $u\to0$ as $q\to\infty$. Since $0<\tau\leq1$,
$(1-\alpha)\tau^r+\alpha\tau=u(\alpha\tau+1-\alpha)\leq u$.
Thus $\alpha\tau\leq u$ and $\tau=O_\alpha(u)$.
Rewriting the equation as
$\alpha\tau=(1-\alpha)u+\alpha u\tau-(1-\alpha)\tau^r$, the last two
terms are $O_{\alpha,r}(u^2)$. Hence
$\tau=((1-\alpha)/\alpha)u+O_{\alpha,r}(u^2)$. Raising this expansion
to the $(r-1)$st power and substituting into
$\rho_{\alpha,r}(H_{m,q})=q^{r-1}(\alpha+(1-\alpha)\tau^{r-1})$ gives
\[
 \rho_{\alpha,r}(H_{m,q})
 =\alpha q^{r-1}
 +\frac{(1-\alpha)^r m^{r-1}}{\alpha^{r-1}}
 +O_{\alpha,r,m}(q^{-1}),
\]
which completes the proof.
\end{proof}

For a general complete $r$-partite graph $G=K_{n_1,\ldots,n_r}$, put
\[
 N=\prod_{i=1}^r n_i,
 \qquad d_i=\prod_{j\ne i}n_j=\frac{N}{n_i}.
\]
Here $d_i$ is the $r$-clique degree of a vertex in part $i$.
The ordering gives $d_1\geq\cdots\geq d_r$.
The connected $r$-clique hypergraph $\mathcal H_r(G)$ has an $\alpha$-Perron vector
invariant under automorphisms and hence constant on each part. Write its coordinate
on part $i$ as $a_i>0$.
Equation~\eqref{eq:clique-coordinate} gives
\[
 \bigl(\rho_{\alpha,r}(G)-\alpha d_i\bigr)a_i^{r-1}
 =(1-\alpha)d_i\prod_{j\ne i}a_j.
\]
By~\eqref{eq:degree-bounds},
$\rho_{\alpha,r}(G)>\alpha d_1$. Multiplying the equations and
cancelling the positive coordinates gives
\begin{equation}\label{eq:multipartite-characteristic}
 \prod_{i=1}^r\bigl(\rho_{\alpha,r}(G)-\alpha d_i\bigr)
 =(1-\alpha)^rN^{r-1}.
\end{equation}
Let $f(t)=\prod_{i=1}^r(t-\alpha d_i)$. Then $f$ is strictly increasing
for $t>\alpha d_1$. Hence~\eqref{eq:multipartite-characteristic}
has exactly one root in $(\alpha d_1,\infty)$.

\begin{lemma}\label{lem:root-comparison}
Let $r\geq3$, $0<\alpha<1$, and let $1\leq n_1\leq\cdots\leq n_r$ be integers.
Put $d_i=\prod_{j\ne i}n_j$ and $N=\prod_i n_i$. If a real number $t$
satisfies
\[
 t>\alpha d_1,
 \qquad
 \prod_i(t-\alpha d_i)-(1-\alpha)^rN^{r-1}<0,
\]
then $t<\rho_{\alpha,r}(K_{n_1,\ldots,n_r})$.
\end{lemma}

\begin{proof}
On the interval $t>\alpha d_1$, every factor in the product is
positive and strictly increasing in $t$. At
$t=\rho_{\alpha,r}(K_{n_1,\ldots,n_r})$, the product equals
$(1-\alpha)^rN^{r-1}$ by~\eqref{eq:multipartite-characteristic}.
The product at $t$ is smaller, so
$t<\rho_{\alpha,r}(K_{n_1,\ldots,n_r})$.
\end{proof}

A set $D\subseteq\mathbb R^r$ is symmetric if it is invariant under
coordinate permutations. A function on $D$ is symmetric if its value
is unchanged by these permutations. For $x,y\in\mathbb R^r$, let
$x_{[1]}\geq\cdots\geq x_{[r]}$ and
$y_{[1]}\geq\cdots\geq y_{[r]}$ denote their coordinates in
nonincreasing order. We write $x\prec y$ and say that $x$ is majorized by $y$ if
$\sum_{i=1}^k x_{[i]}\leq\sum_{i=1}^k y_{[i]}$ for
$1\leq k<r$ and $\sum_i x_i=\sum_i y_i$.
A function $f:D\to\mathbb R$ is Schur-convex if $x\prec y$ implies
$f(x)\leq f(y)$ for all $x,y\in D$.
It is strictly Schur-convex if, in addition, $f(x)<f(y)$ whenever
$x\prec y$ and $x$ is not a coordinate permutation of $y$.

We use the Schur--Ostrowski criterion in
\cite[Chapter~3, A.4 and~A.4.a]{MarshallOlkinArnold}, together with the
strict form obtained from A.2.a and~A.5 there.

\begin{lemma}[\cite{MarshallOlkinArnold}]
\label{lem:schur-ostrowski}
Let $D\subseteq\mathbb R^r$ be a symmetric convex set with nonempty interior,
and let $f:D\to\mathbb R$ be continuous and symmetric on $D$.
Assume that $f$ is continuously differentiable in the interior of $D$.
Then $f$ is Schur-convex on $D$ if and only if
\[
(x_i-x_j)\left(
\frac{\partial f}{\partial x_i}(x)
-\frac{\partial f}{\partial x_j}(x)
\right)\geq0
\]
for every $x$ in the interior of $D$ and all $i,j$.
If $f$ extends to a continuously differentiable function on an open
neighborhood of $D$ and the inequality is strict for every $x\in D$
whenever $x_i\ne x_j$, then $f$ is strictly Schur-convex on $D$.
\end{lemma}

\noindent\begin{minipage}{\linewidth}
\setlength{\parindent}{1.5em}
To bound $n^{1-r}\rho_{c_r,r}(K_{n_1,\ldots,n_r})$ in
Proposition~\ref{prop:multipartite-phase}, we need the following product inequality.

\begin{lemma}
\label{lem:multipartite-product}
Let $r\geq3$, and let $p_1,\ldots,p_r>0$ satisfy $\sum_{i=1}^r p_i=1$.
Set $B_r=r^{1-r}$, $c_r=((r-1)/r)^{r-1}$, and
$\widehat d_i=\prod_{j\ne i}p_j$.
Then each factor on the left-hand side below is positive, and
\begin{equation}\label{eq:multipartite-product}
\prod_{i=1}^r
\frac{B_r/\widehat d_i-c_r}{1-c_r}
\geq1.
\end{equation}
Equality holds if and only if $p_1=\cdots=p_r=1/r$.
\end{lemma}
\end{minipage}

\begin{proof}
Put $y_i=rp_i$ and $Y=\prod_{i=1}^r y_i$. Then $\sum_i y_i=r$ and
$B_r/\widehat d_i=y_i/Y$, so it suffices to prove
\[
 \prod_{i=1}^r\left(\frac{y_i}{Y}-c_r\right)
 \ge(1-c_r)^r.
\]
We will show that $\log\prod_{i=1}^r(y_i/Y-c_r)$ is strictly Schur-convex,
so it is minimized at $y=(1,\ldots,1)$ under $\sum_i y_i=r$.
Since Lemma~\ref{lem:schur-ostrowski} requires a domain with nonempty
interior, we work on the larger symmetric convex set
\[
 \Omega=\left\{y\in\mathbb R_{>0}^r:\sum_{i=1}^r y_i\leq r\right\},
\]
which contains every $y>0$ with $\sum_i y_i=r$ and has nonempty interior.
For $y\in\Omega$, set $z_i=c_rY/y_i=c_r\prod_{j\ne i}y_j$.
The arithmetic--geometric mean (AM--GM) inequality gives
\[
 z_i=c_r\prod_{j\ne i}y_j
 \leq c_r\left(\frac{r-y_i}{r-1}\right)^{r-1}<1.
\]
Thus every factor $y_i/Y-c_r=(y_i/Y)(1-z_i)$ is positive.
We can therefore define
\[
 L(y)=\log\prod_{i=1}^r\left(\frac{y_i}{Y}-c_r\right)
 =-(r-1)\log Y+\sum_{i=1}^r\log(1-z_i).
\]
The strict inequalities $y_i>0$ and $z_i<1$ also show that $L$ extends
smoothly to an open neighborhood of $\Omega$.

To compare two partial derivatives, we need one further bound.
For $i\ne j$, the $r-1$ positive numbers $y_i+y_j$ and $y_\ell$ with
$\ell\ne i,j$ have sum at most $r$. Another application of AM--GM gives
\[
 z_i+z_j
 =c_r(y_i+y_j)\prod_{\ell\ne i,j}y_\ell
 \leq c_r\left(\frac r{r-1}\right)^{r-1}=1.
\]
Since $\partial z_i/\partial y_i=0$ and
$\partial z_j/\partial y_i=z_j/y_i$ for $j\ne i$, we obtain
\[
 \frac{\partial L}{\partial y_i}
 =-\frac{r-1}{y_i}-\frac1{y_i}\sum_{j\ne i}\frac{z_j}{1-z_j}
 =-\frac1{y_i}\sum_{j\ne i}\frac1{1-z_j}.
\]
Using $y_i z_i=y_j z_j=c_rY$, subtraction and simplification give
\[
(y_i-y_j)
\left(
\frac{\partial L}{\partial y_i}
-
\frac{\partial L}{\partial y_j}
\right)
=
\frac{(y_i-y_j)^2}{y_iy_j}
\left(
\sum_{k\ne i,j}\frac1{1-z_k}
+\frac{1-z_i-z_j}{(1-z_i)(1-z_j)}
\right).
\]
The sum in parentheses is positive: it has $r-2\geq1$ positive terms,
and the last term is nonnegative because $z_i+z_j\leq1$.
Hence the derivative condition is strict whenever $y_i\ne y_j$.
Lemma~\ref{lem:schur-ostrowski} now shows that $L$ is strictly
Schur-convex on $\Omega$.

Finally, if $\sum_i y_i=r$, the equal-coordinate vector satisfies
$(1,\ldots,1)\prec y$. Strict Schur-convexity therefore gives
\[
 L(y)\geq L(1,\ldots,1)=r\log(1-c_r),
\]
with equality if and only if $y=(1,\ldots,1)$.
Since $\exp L(y)=\prod_{i=1}^r(y_i/Y-c_r)$, substituting $y_i=rp_i$ gives
\[
 \prod_{i=1}^r\frac{B_r/\widehat d_i-c_r}{1-c_r}
 =\frac{\exp L(y)}{(1-c_r)^r}\geq1,
\]
with equality if and only if $p_i=1/r$ for every $i$.
\end{proof}

Fix an integer $r\geq3$ and $0<\alpha<1$, and put
$B_r=r^{1-r}$ and $M_\alpha=\max\{B_r,\alpha(r-1)^{1-r}\}$.
For $G=K_{n_1,\ldots,n_r}$ on $n=\sum_i n_i$ vertices, put
$p_i=n_i/n$ and $p=(p_1,\ldots,p_r)$.
Thus $p_i$ is the proportion of vertices in part $i$.
Let $\Delta_r=\{p\in\mathbb R^r:p_i\geq0\text{ for every }i,
\ \sum_{i=1}^r p_i=1\}$.
Put $\Delta_r^\uparrow=\{p\in\Delta_r:p_1\leq\cdots\leq p_r\}$,
$b_r=(1/r,\ldots,1/r)$, and
$q_r=(0,1/(r-1),\ldots,1/(r-1))$. Define
\[
 \mathcal E_\alpha=
 \begin{cases}
  \{b_r\},&0<\alpha<c_r,\\
  \{b_r,q_r\},&\alpha=c_r,\\
  \{q_r\},&c_r<\alpha<1,
 \end{cases}
 .
\]
The Euclidean norm of $z=(z_1,\ldots,z_r)\in\mathbb R^r$ is
$\|z\|_2=(\sum_{i=1}^r z_i^2)^{1/2}$.
For a nonempty set $S\subseteq\mathbb R^r$, the distance from $p$ to $S$ is
$\operatorname{dist}(p,S)=\inf_{q\in S}\|p-q\|_2$.

\begin{proposition}
\label{prop:multipartite-phase}
Let $r\geq3$ be an integer and $0<\alpha<1$.
Then
\begin{equation}\label{eq:multipartite-asymptotic}
 \lim_{n\to\infty}n^{1-r}
 \max_{\substack{n_1+\cdots+n_r=n\\1\leq n_1\leq\cdots\leq n_r}}
 \rho_{\alpha,r}(K_{n_1,\ldots,n_r})
 =M_\alpha.
\end{equation}
Moreover, let $(G_n)_{n\geq r}$ be a sequence of $n$-vertex complete
$r$-partite graphs. For each $n$, write $G_n=K_{n_1,\ldots,n_r}$
and put $p_n=(n_1/n,\ldots,n_r/n)$.
If $(G_n)_{n\geq r}$ is asymptotically extremal, meaning
$n^{1-r}\rho_{\alpha,r}(G_n)\to M_\alpha$ as $n\to\infty$, then
\[
 \operatorname{dist}\bigl(p_n,\mathcal E_\alpha\bigr)\longrightarrow0.
\]
\end{proposition}

\begin{proof}
Let $G=K_{n_1,\ldots,n_r}$, let $V_i$ be part $i$, and put $p_i=n_i/n$.
Let $x$ be the $\alpha$-Perron vector and write $a_i>0$ for its common
coordinate on $V_i$, as in the derivation of~\eqref{eq:multipartite-characteristic}.
Set $X_i=\sum_{v\in V_i}x_v^r=n_i a_i^r$.
Then $p\in\Delta_r^\uparrow$ and $X\in\Delta_r$.
Set
\[
 D(p,X)=\sum_{i=1}^rX_i\prod_{j\ne i}p_j,\qquad
 C(p,X)=r\left(\prod_{i=1}^rp_i\right)^{(r-1)/r}
          \left(\prod_{i=1}^rX_i\right)^{1/r}.
\]
Conversely, every $X\in\Delta_r$ is realized by
$x_v=(X_i/n_i)^{1/r}$ for $v\in V_i$, with $\sum_vx_v^r=1$. Thus
equations~\eqref{eq:hypergraph-variational} and~\eqref{eq:variational} give
\begin{equation}\label{eq:general-multipartite-variation}
\frac{\rho_{\alpha,r}(K_{n_1,\ldots,n_r})}{n^{r-1}}
 =
 \max_{\substack{X_i\geq0\\\sum_iX_i=1}}
 \left\{\alpha D(p,X)+(1-\alpha)C(p,X)\right\}.
\end{equation}
Here $p$ is fixed by the ordered part sizes, and $X$ ranges over the
full simplex $\Delta_r$.

For $p\in\Delta_r^\uparrow$ and $0\leq\alpha\leq1$, write
$\widehat d_i=\prod_{j\ne i}p_j$ and define
\[
 f(\alpha,p)=\max_{X\in\Delta_r}
 \{\alpha D(p,X)+(1-\alpha)C(p,X)\}.
\]
The functions $D$ and $C$ are continuous on the compact set
$\Delta_r^\uparrow\times\Delta_r$, so $f(\alpha,p)$ is continuous in $p$
for fixed $\alpha$.
For fixed $p$, $f(\alpha,p)$ is convex in $\alpha$: each
$\alpha D(p,X)+(1-\alpha)C(p,X)$ is affine in $\alpha$, and a pointwise
maximum of affine functions is convex.
We will bound $f(0,p)$, $f(c_r,p)$, and $f(1,p)$, then use this convexity
to treat the remaining values of $\alpha$.

Since $p_1\leq\cdots\leq p_r$, we have $\widehat d_1\geq\cdots\geq\widehat d_r$.
Suppose $0<\alpha<1$ and $p_1>0$.
For $p=(p_1,\ldots,p_r)$ with positive rational coordinates, choose
integer part sizes satisfying $n_i/n=p_i$.
Dividing~\eqref{eq:multipartite-characteristic} by $n^{r(r-1)}$ gives
\begin{equation}\label{eq:normalized-characteristic}
 \prod_{i=1}^r\bigl(f(\alpha,p)-\alpha\widehat d_i\bigr)
 =(1-\alpha)^r\left(\prod_{i=1}^r p_i\right)^{r-1}.
\end{equation}
By density of rational points and continuity of $f(\alpha,p)$ and the products
in $p$, \eqref{eq:normalized-characteristic} holds whenever all $p_i>0$.
Taking $X=(1,0,\ldots,0)$ gives $f(\alpha,p)\geq\alpha\widehat d_1$.
If $f(\alpha,p)=\alpha\widehat d_1$, the product on the left
would be zero, while the right-hand side is positive. Hence
$f(\alpha,p)>\alpha\widehat d_1$.

At $\alpha=c_r$, Lemma~\ref{lem:multipartite-product} gives
$B_r>c_r\widehat d_1$ and
\[
 \prod_{i=1}^r(B_r-c_r\widehat d_i)
 \geq(1-c_r)^r\left(\prod_{i=1}^r p_i\right)^{r-1},
\]
with equality if and only if $p_i=1/r$ for every $i$.
The product $\prod_i(t-c_r\widehat d_i)$ is strictly increasing for
$t>c_r\widehat d_1$. Comparing with~\eqref{eq:normalized-characteristic}
at $\alpha=c_r$ gives
$f(c_r,p)\leq B_r$, with equality only at $p=b_r$.

For every $p\in\Delta_r^\uparrow$, maximizing $C(p,X)$ and $D(p,X)$ over
$X\in\Delta_r$ gives
\[
 f(0,p)=\left(\prod_{i=1}^r p_i\right)^{(r-1)/r}\leq B_r,
 \qquad
 f(1,p)=\widehat d_1\leq(r-1)^{1-r}.
\]
The bound $f(0,p)\leq B_r$ follows from AM--GM, with equality only at
$p=b_r$. To bound $f(1,p)$, use
$\widehat d_1\leq((1-p_1)/(r-1))^{r-1}\leq(r-1)^{1-r}$.
Equality holds if and only if $p_1=0$ and
$p_2=\cdots=p_r=1/(r-1)$, that is, $p=q_r$.
If $p_1=0$, then $C(p,X)=0$ for every $X\in\Delta_r$, so
\[
 f(\alpha,p)=\alpha\widehat d_1.
\]
Since $B_r=c_r(r-1)^{1-r}$, we also have $f(c_r,p)\leq B_r$ for $p_1=0$,
with equality only at $p=q_r$.

For fixed $p$, convexity of $f(\alpha,p)$ in $\alpha$ gives
\[
 f(\alpha,p)\leq
 \begin{cases}
  (1-\alpha/c_r)f(0,p)+(\alpha/c_r)f(c_r,p),
     &0\leq\alpha\leq c_r,\\[2mm]
  \dfrac{1-\alpha}{1-c_r}f(c_r,p)
   +\dfrac{\alpha-c_r}{1-c_r}f(1,p),
     &c_r\leq\alpha\leq1.
 \end{cases}
\]
Substituting the bounds for $f(0,p)$, $f(c_r,p)$, and $f(1,p)$ yields
\[
 f(\alpha,p)
 \leq
 \begin{cases}
  B_r,&0\leq\alpha\leq c_r,\\[2mm]
  \alpha(r-1)^{1-r},&c_r\leq\alpha\leq1.
 \end{cases}
\]
For $0<\alpha<c_r$, equality $f(\alpha,p)=B_r$ requires $f(0,p)=B_r$,
so $p=b_r$. For $c_r<\alpha<1$, equality
$f(\alpha,p)=\alpha(r-1)^{1-r}$ requires $f(1,p)=(r-1)^{1-r}$,
so $p=q_r$. At $\alpha=c_r$, equality holds exactly when
$p\in\{b_r,q_r\}$.
At $p=b_r$, we have $D(b_r,X)=B_r$ and $C(b_r,X)\leq B_r$
for every $X\in\Delta_r$ by AM--GM, with equality in the latter when $X=b_r$.
Hence $f(\alpha,b_r)=B_r$.
At $p=q_r$, the identity $f(\alpha,p)=\alpha\widehat d_1$
gives $f(\alpha,p)=\alpha(r-1)^{1-r}$.

Thus, we have proved
\[
 f(\alpha,p)=M_\alpha\iff p\in\mathcal E_\alpha.
\]

The $r$-clique hypergraph $\mathcal H_r(K_{n_1,\ldots,n_r})$ has average degree
$r\prod_i n_i/n$. For the part sizes $n_1,\ldots,n_r$ of $T_r(n)$,
$p_i\to1/r$ as $n\to\infty$, so~\eqref{eq:degree-bounds} gives
\[
 \frac{\rho_{\alpha,r}(T_r(n))}{n^{r-1}}
 \geq r\prod_i p_i\longrightarrow B_r.
\]
For $K_1\vee T_{r-1}(n-1)$, we have $n_1=1$ and
$p_i\to1/(r-1)$ for $i\geq2$. The vertex in part $1$ has $r$-clique
degree $\prod_{j=2}^r n_j$, so~\eqref{eq:degree-bounds} gives
\[
 \frac{\rho_{\alpha,r}(K_1\vee T_{r-1}(n-1))}{n^{r-1}}
 >\alpha\prod_{j=2}^r p_j\longrightarrow\alpha(r-1)^{1-r}.
\]
Using $T_r(n)$ and $K_1\vee T_{r-1}(n-1)$ in~\eqref{eq:multipartite-asymptotic}
gives a limit inferior of at least $M_\alpha$.
On the other hand,~\eqref{eq:general-multipartite-variation} gives
$n^{1-r}\rho_{\alpha,r}(K_{n_1,\ldots,n_r})=f(\alpha,p)\leq M_\alpha$
for every choice of the part sizes. This proves~\eqref{eq:multipartite-asymptotic}.

It remains to prove $\operatorname{dist}(p_n,\mathcal E_\alpha)\to0$.
For every $\varepsilon>0$, the points in $\Delta_r^\uparrow$ at distance at least $\varepsilon$
from $\mathcal E_\alpha$ form a compact set on which $f(\alpha,p)<M_\alpha$.
If this set is nonempty, $M_\alpha-f(\alpha,p)$ has a positive minimum on it;
if it is empty, any positive gap works. Hence there exists
$\gamma=\gamma(r,\alpha,\varepsilon)>0$ such that, for $p\in\Delta_r^\uparrow$,
\begin{equation}\label{eq:phase-uniform-gap}
 \operatorname{dist}(p,\mathcal E_\alpha)\geq\varepsilon
 \quad\Longrightarrow\quad
 f(\alpha,p)\leq M_\alpha-\gamma.
\end{equation}
For an asymptotically extremal sequence,
$n^{1-r}\rho_{\alpha,r}(G_n)\to M_\alpha$. By~\eqref{eq:phase-uniform-gap},
$\operatorname{dist}(p_n,\mathcal E_\alpha)\to0$.
\end{proof}

\begin{proof}[Proof of Theorem~\ref{thm:global-phase}]
Theorem~\ref{thm:multipartite-reduction} reduces the maximum to complete
$r$-partite graphs. Proposition~\ref{prop:multipartite-phase} therefore gives
\[
 \lim_{n\to\infty}n^{1-r}
 \max_{\substack{|V(G)|=n\\G\text{ is }K_{r+1}\text{-free}}}
 \rho_{\alpha,r}(G)
 =M_\alpha.\qedhere
\]
\end{proof}

\begin{proof}[Proof of Corollary~\ref{cor:balanced-multipartite}]
Theorem~\ref{thm:multipartite-reduction} reduces the maximization to
complete $r$-partite graphs. For $p_i=n_i/n>0$,
the bound and equality cases established in the proof of
Proposition~\ref{prop:multipartite-phase}
give $n^{1-r}\rho_{\alpha,r}(K_{n_1,\ldots,n_r})\leq B_r$, with equality
only when $p_1=\cdots=p_r=1/r$.
Since $r\mid n$, every part has size $n/r$.

Conversely,
$\mathcal H_r(T_r(n))$ is $(n/r)^{r-1}$-regular: an $r$-clique containing
a fixed vertex is obtained by choosing one vertex from each other part.
Equation~\eqref{eq:degree-bounds} therefore gives
$\rho_{\alpha,r}(T_r(n))=(n/r)^{r-1}$, so $T_r(n)$ attains the bound.
\end{proof}

\subsection{Exact optimization of the part sizes}

For fixed $r$ and $\alpha$, we use Proposition~\ref{prop:multipartite-phase}
to study maximizing integer partitions $(n_1,\ldots,n_r)$ with
$p=(n_1/n,\ldots,n_r/n)$ near $\mathcal E_\alpha$.
The first coordinate of $q_r$ is zero, but $p_1>0$ for every graph.
At $\alpha=c_r$, we have $f(c_r,b_r)=f(c_r,q_r)=B_r$, so we compare
lower-order terms.

We next estimate
$\rho_{\alpha,r}(K_{m,n_2,\ldots,n_r})-\alpha\prod_{i=2}^r n_i$
for fixed $m$. The case $m=1$ will be used in
Lemma~\ref{lem:eventual-part-sizes} to prove $n_r-n_2\leq1$ when $\alpha>c_r$.
For the proof of Lemma~\ref{lem:uniform-correction}, we first record
the following identity, valid for real $a,x_2,\ldots,x_r$:
\begin{equation}\label{eq:product-difference}
\begin{aligned}
 &\sum_{i=2}^ra^{i-2}(x_i-a)\prod_{j=i+1}^r x_j\\
 &\quad=\sum_{i=2}^ra^{i-2}\prod_{j=i}^r x_j
       -\sum_{i=2}^ra^{i-1}\prod_{j=i+1}^r x_j\\
 &\quad=\prod_{j=2}^r x_j
       +\sum_{i=3}^ra^{i-2}\prod_{j=i}^r x_j
       -\sum_{i=3}^ra^{i-2}\prod_{j=i}^r x_j
       -a^{r-1}\\
 &\quad=\prod_{j=2}^r x_j-a^{r-1}.
\end{aligned}
\end{equation}

\begin{lemma}\label{lem:uniform-correction}
Fix $r\geq3$, $0<\alpha<1$, and a positive integer $m$. Put
$\beta=1-\alpha$. Define $C_1=C_1(r,\alpha,m)$ and $C_2=C_2(r,\alpha,m)$ by
\[
 C_1=\beta^r(2/\alpha)^{r-1}m^{r-1},
 \qquad
 C_2
 =\frac{C_1(r-1)(C_1+\alpha m)
              (1+\alpha+C_1)^{r-1}}{\alpha^{r-1}}.
\]
For integers $m\leq n_2\leq\cdots\leq n_r$, let $P=\prod_{i=2}^r n_i$
and $L=n_2$. If $L\geq2m$, then
$\delta=\rho_{\alpha,r}(K_{m,n_2,\ldots,n_r})-\alpha P$ satisfies
\begin{equation}\label{eq:uniform-correction-rate}
 0<\delta\leq C_1,
 \qquad
 \left|\delta-\frac{\beta^r m^{r-1}}{\alpha^{r-1}}\right|
 \leq\frac{C_2}{L}.
\end{equation}
\end{lemma}

\begin{proof}
Equation~\eqref{eq:multipartite-characteristic} gives $\delta>0$ and
\[
 \delta\prod_{i=2}^r
 \left(\alpha+\frac{\delta}{P}-\frac{\alpha m}{n_i}\right)
 =\beta^r m^{r-1}.
\]
Since $n_i\geq L\geq2m$ for $i\geq2$, every factor is at least
$\alpha/2$, so $\delta\leq C_1$.
Put $M=1+\alpha+C_1$ and
$v_i=\alpha+\delta/P-\alpha m/n_i$. Since $P\geq L^{r-1}\geq L$, we have
\[
 0<v_i\leq M,
 \qquad |v_i-\alpha|\leq\frac{C_1+\alpha m}{L}.
\]
Applying~\eqref{eq:product-difference} with $a=\alpha$ and $x_i=v_i$,
and using $M\geq1$, we obtain
\[
 \left|\prod_{i=2}^r v_i-\alpha^{r-1}\right|
 \leq\frac{(r-1)(C_1+\alpha m)M^{r-1}}{L}.
\]
Finally,
\[
 \left|\delta-\frac{\beta^r m^{r-1}}{\alpha^{r-1}}\right|
 =\frac{\delta}{\alpha^{r-1}}
   \left|\alpha^{r-1}-\prod_{i=2}^r v_i\right|
 \leq\frac{C_2}{L}.
\]
This proves~\eqref{eq:uniform-correction-rate}.
\end{proof}

\begin{samepage}
\begin{lemma}\label{lem:eventual-part-sizes}
Let $r\geq3$ and $0<\alpha<1$, and put $c_r=((r-1)/r)^{r-1}$.
There exists an integer $n_0=n_0(r,\alpha)\geq r$ such that, for $n\geq n_0$
and integers $1\leq n_1\leq\cdots\leq n_r$ with $n_1+\cdots+n_r=n$,
the following hold.
\begin{enumerate}[label=\textup{(\roman*)}]
\item If $0<\alpha\leq c_r$, then
\[
 \rho_{\alpha,r}(K_{n_1,\ldots,n_r})
 \leq\rho_{\alpha,r}(T_r(n)),
\]
with equality if and only if $K_{n_1,\ldots,n_r}\cong T_r(n)$.
\item If $c_r<\alpha<1$, then
\[
 \rho_{\alpha,r}(K_{n_1,\ldots,n_r})
 \leq\rho_{\alpha,r}\bigl(K_1\vee T_{r-1}(n-1)\bigr),
\]
with equality if and only if
$K_{n_1,\ldots,n_r}\cong K_1\vee T_{r-1}(n-1)$.
\end{enumerate}
\end{lemma}
\end{samepage}

\begin{proof}
Fix $r$ and $\alpha$, and put $B_r=r^{1-r}$ and $\beta=1-\alpha$.
Let $n\geq r$ be an integer, and choose positive integers
$n_1\leq\cdots\leq n_r$ with $n_1+\cdots+n_r=n$ such that
$G=K_{n_1,\ldots,n_r}$ has maximum $\alpha$-$r$-clique spectral radius
among all $n$-vertex complete $r$-partite graphs.
Write $p=(p_1,\ldots,p_r)$, where $p_i=n_i/n$, and put
$\widehat d_i=\prod_{j\ne i}p_j$, $\rho=\rho_{\alpha,r}(G)$, and
$\widehat\rho=\rho/n^{r-1}$.

To determine the part sizes of $G$ for all sufficiently large $n$, we use
Proposition~\ref{prop:multipartite-phase}: for every fixed open neighborhood
$U$ of $\mathcal E_\alpha$ in $\Delta_r^\uparrow$, there is an integer
$N=N(r,\alpha,U)$ such that every such graph $G$ has $p\in U$ whenever
$n\geq N$.
We examine $p$ near $b_r$ and $q_r$ separately and choose the neighborhoods
in terms of $r$ and $\alpha$ before fixing a common lower bound on $n$.

For $0<\alpha\leq c_r$, consider such graphs $G$ with $p$ near $b_r$.

\nopagebreak[4]
\noindent\begin{minipage}{\linewidth}
\setlength{\parindent}{1.5em}
To compare parts $r$ and $1$, set $a=p_r=n_r/n$, $b=p_1=n_1/n$,
and $\xi=(\xi_1,\ldots,\xi_{r-2})=(p_2,\ldots,p_{r-1})$.
Put $S=a+b$, $R=\prod_{\ell=1}^{r-2}\xi_\ell$,
and $t=ab$.
\end{minipage}

A vertex move from part $r$ to part $1$ preserves $S$ and $\xi$.
To compare the spectral radii by Lemma~\ref{lem:root-comparison},
we keep $\rho$ fixed. Dividing~\eqref{eq:multipartite-characteristic} by
$n^{r(r-1)}$ and expanding the product of the factors for parts $1$ and $r$
gives $\Phi(\widehat\rho,S,\xi,t)=0$, where
\[
 \Phi(\widehat\rho,S,\xi,t)
 =(\widehat\rho^2-\alpha\widehat\rho SR+\alpha^2tR^2)
   \prod_{\ell=1}^{r-2}
   \left(\widehat\rho-\frac{\alpha tR}{\xi_\ell}\right)
 -\beta^r(tR)^{r-1}.
\]
Suppose $n_r\geq n_1+2$. Move one vertex from part $r$ to part $1$.
Before reordering, the two changed fractions are $a'=a-h$ and
$b'=b+h$, where $h=1/n$, so
\[
 t'=a'b'=t+h(a-b-h)>t.
\]
Write $n'_1\leq\cdots\leq n'_r$ for the new ordered part sizes and put
$\rho'=\rho_{\alpha,r}(K_{n'_1,\ldots,n'_r})$.

Write $\Phi_t=\partial\Phi/\partial t$. At $p=b_r$, put $u=1/r$.
Direct differentiation gives
\[
 \Phi_t\bigl(u^{r-1},2u,(u,\ldots,u),u^2\bigr)
 =\beta^{r-2}u^{(r-2)(r+1)}\bigl(r\alpha-(r-1)\bigr)<0,
\]
since $\alpha\leq c_r<(r-1)/r$. At $p=b_r$,
$\widehat\rho-\alpha\max_i\widehat d_i=(1-\alpha)B_r>0$.

By continuity, $\Phi_t<0$ on a fixed open box around
$(\widehat\rho,S,\xi,t)=(u^{r-1},2u,(u,\ldots,u),u^2)$
and $\widehat\rho>\alpha\max_i\widehat d_i$ whenever
$(\widehat\rho,p)$ is sufficiently close to $(B_r,b_r)$.

Since $\widehat\rho\to B_r$ for these graphs and
$|a'-a|=|b'-b|=1/n$, there are a fixed neighborhood
$V$ of $b_r$ in $\Delta_r^\uparrow$ and an integer $N_1=N_1(r,\alpha)$
such that, for every
such graph $G$ with $n\geq N_1$ and $p\in V$, both
$(\widehat\rho,S,\xi,t)$ and $(\widehat\rho,S,\xi,t')$ lie in this box.
The inequality $\widehat\rho>\alpha\max_i\widehat d_i$ holds for the
proportions before and after the move, with $\widehat\rho$ held at
its value before the move. The segment between $t$ and $t'$
also lies in the box, since $\widehat\rho,S,\xi$ stay fixed.

Consequently, $\Phi(\widehat\rho,S,\xi,t')<\Phi(\widehat\rho,S,\xi,t)=0$.
The product in~\eqref{eq:multipartite-characteristic} is invariant
under reordering. The preceding positivity condition gives
$\rho>\alpha\prod_{i=2}^r n'_i$.
Lemma~\ref{lem:root-comparison}, applied at $n^{r-1}\widehat\rho=\rho$,
gives $\rho'>\rho$, contradicting maximality.
Thus, for $n\geq N_1(r,\alpha)$, every such graph $G$ with
$p\in V$ satisfies $n_r-n_1\leq1$.

For $c_r\leq\alpha<1$, consider such graphs $G$ with $p$ near $q_r$.
Put $c_0=1/(2(r-1))$ and $C_{r,\alpha}=\beta^r(2/\alpha)^{r-1}$, and take
$0<\eta\leq1/(4(r-1))$.
Take the neighborhood $V_q(\eta)$ of $q_r$ defined by
\[
 V_q(\eta)=\{p\in\Delta_r^\uparrow:p_1<\eta,\
 |p_i-1/(r-1)|<\eta\ (2\leq i\leq r)\}.
\]
Write $m=n_1$, $P=\prod_{i=2}^r n_i$, and
$\rho=\alpha P+\delta$. For $p\in V_q(\eta)$,
$m\leq\eta n$, $n_i\geq c_0 n$, and $m/n_i\leq1/2$.
Equation~\eqref{eq:multipartite-characteristic} gives
\[
 \delta\prod_{i=2}^r
 \left(\alpha+\frac{\delta}{P}-\frac{\alpha m}{n_i}\right)
 =\beta^r m^{r-1}.
\]
Here $\delta>0$, and $m/n_i\leq1/2$ makes every factor in the product
at least $\alpha/2$. Therefore
\[
 0<\delta\leq C_{r,\alpha}m^{r-1}.
\]
The bound also holds for $m=m(n)$, since $C_{r,\alpha}$ depends only
on $r$ and~$\alpha$.

First suppose $\alpha>c_r$.
If $m\geq2$, move $m-1$ vertices from part $1$ to part $r$.
The resulting graph has ordered part sizes
$(1,n_2,\ldots,n_{r-1},n_r+m-1)$.
The product of the last $r-1$ part sizes is
\[
 \widetilde P=P+(m-1)\prod_{i=2}^{r-1} n_i.
\]
Set $\widetilde\rho=\rho_{\alpha,r}(K_{1,n_2,\ldots,n_{r-1},n_r+m-1})$.
By~\eqref{eq:degree-bounds}, $\widetilde\rho>\alpha\widetilde P$.
\begin{samepage}
Since $m\leq2(m-1)$ and $m/n\leq\eta$, we obtain
\[
 \widetilde\rho-\rho
 >\alpha(m-1)\prod_{i=2}^{r-1} n_i-C_{r,\alpha}m^{r-1}
 \geq\bigl(\alpha c_0^{r-2}-2C_{r,\alpha}\eta^{r-2}\bigr)
       (m-1)n^{r-2}.
\]
We now fix $\eta$ small enough that
$2C_{r,\alpha}\eta^{r-2}\leq\alpha c_0^{r-2}/2$.
Then $\widetilde\rho>\rho$, contradicting maximality.
Thus every such graph $G$ with $p\in V_q(\eta)$ has $n_1=m=1$.
\end{samepage}

It remains to prove $n_r-n_2\leq1$.
Put $C_2=C_2(r,\alpha,1)$ from
Lemma~\ref{lem:uniform-correction} and $L=n_2\geq c_0 n$.
Equation~\eqref{eq:uniform-correction-rate} gives
\[
 \left|\delta-\frac{\beta^r}{\alpha^{r-1}}\right|\leq\frac{C_2}{L}.
\]
If $n_r\geq n_2+2$, move one vertex from part $r$ to part $2$.
After reordering, write $1=n'_1\leq n'_2\leq\cdots\leq n'_r$,
$P'=\prod_{i=2}^r n'_i$, and $L'=n'_2$. Every one of the last $r-1$
parts still has size at least $n_2$, so $L'\geq L\geq c_0 n$.
Equation~\eqref{eq:uniform-correction-rate} applies both before and
after the move, and
\[
 P'-P=(n_r-n_2-1)\prod_{i=3}^{r-1}n_i
       \geq(c_0 n)^{r-3}.
\]
Set $\rho'=\rho_{\alpha,r}(K_{1,n'_2,\ldots,n'_r})$ and
$\delta'=\rho'-\alpha P'$. Then
$|\delta'-\delta|\leq C_2/L+C_2/L'\leq2C_2/(c_0 n)$. Hence
\[
 \rho'-\rho\geq\alpha(c_0 n)^{r-3}-\frac{2C_2}{c_0 n}>0
\]
for all sufficiently large $n$. This also holds for $r=3$, when
$\prod_{i=3}^{r-1}n_i$ is empty and equals $1$. Therefore
$n_r-n_2\leq1$.

At $\alpha=c_r$, we will show $\widehat\rho\leq B_r-B_r/(2n)$ near $q_r$,
while $n^{1-r}\rho_{c_r,r}(T_r(n))=B_r+O_r(n^{-2})$.
Take $0<\eta\leq1/(4(r-1))$ with $C_{r,c_r}\eta^{r-2}\leq B_r/2$
and use $V_q(\eta)$. Write $\mu=m/n$. By AM--GM and
$0<\delta\leq C_{r,c_r}m^{r-1}$, every choice of the part sizes with
$p\in V_q(\eta)$ satisfies
\[
 \widehat\rho\leq B_r(1-\mu)^{r-1}+C_{r,c_r}\mu^{r-1}.
\]
Using $0<\mu\leq\eta$, $(1-\mu)^{r-1}\leq1-\mu$, and $m\geq1$, we obtain
\[
 \widehat\rho\leq B_r-\frac{B_r}{2}\mu
 \leq B_r-\frac{B_r}{2n}.
\]
Let $n_1^*\leq\cdots\leq n_r^*$ be the part sizes of $T_r(n)$, and put
$p_i^*=n_i^*/n=1/r+e_i$. Then
$|e_i|\leq1/n$ and $\sum_i e_i=0$.
The terms linear in $e_i$ in the expansion of $r\prod_i(1/r+e_i)$
cancel because $\sum_i e_i=0$. Hence
$r\prod_i p_i^*=B_r+O_r(n^{-2})$. The average-degree bound in
\eqref{eq:degree-bounds}, together with the upper bound established
in the proof of Proposition~\ref{prop:multipartite-phase}, gives
\[
 \frac{\rho_{c_r,r}(T_r(n))}{n^{r-1}}=B_r+O_r(n^{-2}).
\]
For all sufficiently large $n$,
$n^{1-r}\rho_{c_r,r}(T_r(n))>B_r-B_r/(2n)$.
Thus no such graph $G$ has $p\in V_q(\eta)$ at $\alpha=c_r$.

Finally, let $U$ be the union of the neighborhoods chosen above for the
points of $\mathcal E_\alpha$.
Since these neighborhoods depend only on $r$ and $\alpha$, we can choose
an integer $n_0=n_0(r,\alpha)\geq r$ so that $n_0\geq N(r,\alpha,U)$ and
all the preceding comparisons hold whenever $n\geq n_0$.
Then every such graph $G$ has $p\in U$, and the local conclusions give
$n_r-n_1\leq1$ if $\alpha\leq c_r$, or
$n_1=1$ and $n_r-n_2\leq1$ if $\alpha>c_r$.
Hence (i) and (ii), including their equality cases, follow.
\end{proof}

\begin{proof}[Proof of Theorem~\ref{thm:alpha-clique-turan}]
By Theorem~\ref{thm:multipartite-reduction}, all extremal graphs are
complete $r$-partite. For all sufficiently large $n$,
Lemma~\ref{lem:eventual-part-sizes} determines the
part sizes and shows that the extremal graph is unique.
\end{proof}

\clearpage
\section{Open problems}\label{sec:conclusion}

Theorem~\ref{thm:alpha-clique-turan} determines the extremal graphs for
$s=r$ and all sufficiently large $n$.
For $3\leq s<r$, Theorem~\ref{thm:multipartite-reduction} still shows that
every extremal graph is complete $r$-partite, but the maximizing part sizes
remain undetermined. We therefore ask whether $T_r(n)$ and
\mbox{$K_1\vee T_{r-1}(n-1)$} are again the only extremal graphs, and how the
answer depends on $\alpha$.
\begin{samepage}
\begin{problem}\label{prob:part-size}
Let $3\leq s<r$ be fixed integers. For $\alpha\in(0,1)$, determine
\[
 \lim_{n\to\infty}n^{1-s}
 \max_{\substack{n_1+\cdots+n_r=n\\1\leq n_1\leq\cdots\leq n_r}}
 \rho_{\alpha,s}(K_{n_1,\ldots,n_r}),
\]
and the maximizing integer partitions for all sufficiently large $n$.
In particular, for each fixed $\alpha\in(0,1)$ and all sufficiently large
$n$, is every extremal graph either $T_r(n)$ or
$K_1\vee T_{r-1}(n-1)$?
For every sufficiently large $n$, is there a unique
$\alpha_n\in(0,1)$ at which these two graphs have the same $\alpha$-$s$-clique spectral radius,
and does $\alpha_n\to((r-1)/r)^{s-1}$?
\end{problem}
\end{samepage}

For $s=r$, Theorem~\ref{thm:alpha-clique-turan} guarantees the existence
of $n_0(r,\alpha)$ but does not determine its least possible value.
This leads to a second problem concerning the dependence of
$n_0(r,\alpha)$ on $r$ and $\alpha$, particularly near $\alpha=c_r$.

\begin{problem}\label{prob:threshold}
For fixed $r\geq3$ and $0<\alpha<1$, determine the least integer
$n_0(r,\alpha)\geq r$ such that the classification in
Theorem~\ref{thm:alpha-clique-turan} holds for every $n\geq n_0(r,\alpha)$.
How does $n_0(r,\alpha)$ behave as $\alpha$ approaches $c_r$?
\end{problem}

\section*{Funding}
This work is partially supported by the National Natural Science
Foundation of China (No.~12271182).

\section*{Declaration of competing interest}
The authors declare that they have no known competing financial interests
or personal relationships that could have appeared to influence the work
reported in this paper.

\section*{Declaration of generative AI and AI-assisted technologies in the
manuscript preparation process}
During the preparation of this work, OpenAI GPT-6 was used to revise the
language, to assist with checking and revising proofs. The authors are responsible for verifying the mathematical arguments and for the content of the article.

\clearpage
\begingroup
\small
\setlength{\bibsep}{1pt}
\bibliographystyle{abbrvnat}
\bibliography{alpha_clique_tensor}
\endgroup

\end{document}